\documentclass[reqno]{amsart}
\usepackage{graphicx,amsmath,amssymb,hyperref,nicefrac,microtype,xcolor,bm,enumitem}
\usepackage[foot]{amsaddr}
\usepackage[paperwidth=6.11in,paperheight=9.26in,text={4.7in,7.65in},centering,top=.75in]{geometry}

\usepackage[numbers,square,comma]{natbib}

\usepackage{amsthm}
\theoremstyle{plain}

\makeatletter
\newtheorem*{rep@theorem}{\rep@title}
\newcommand{\newreptheorem}[2]{%
\newenvironment{rep#1}[1]{%
 \def\rep@title{#2 \ref{##1}}%
 \begin{rep@theorem}}%
 {\end{rep@theorem}}}
\makeatother

\newtheorem{theorem}{Theorem}

\newreptheorem{theorem}{Theorem}
\newreptheorem{algorithm}{Algorithm}
\newreptheorem{remark}{Remark}
\newreptheorem{lemma}{Lemma}
\newreptheorem{prop}{Proposition}

\numberwithin{equation}{section}

\numberwithin{theorem}{section}

\usepackage{mathpazo}
\usepackage[T1]{fontenc}

\newcommand{\ignore}[1]{}

\begin{document}

\title{Computing the Coefficients for Non-Periodic Highly Oscillatory Orthonormal Functions}

\author{Rockford Sison}
\email{rocky.foster@berkeley.edu}
\address{Department of Mathematics, University of
  Cincinnati}


\keywords{oscillatory integrals, orthogonal functions, 
three term recurrence relations}

\begin{abstract}
  A three term recurrence relation is derived for a basis consisting of polynomials multiplied by sines and cosines with large, but fixed frequencies. A numerical method for computing the coefficients of the three term recurrence relation is derived.
\end{abstract}


\maketitle

\markboth{ROCKFORD SISON}{PRODUCT BASIS
OF LEGENDRE POLYNOMIALS AND OSCILLATORY FUNCTIONS}

\section{Introduction}

Orthogonal functions, typically polynomials or sines and cosines, have a long history in solving problem\cite{gautschi:82}. These orthogonal functions give rise to three-term recurrence relations. We will examine the problem of trying to represent highly oscillatory, non-periodic functions, in the form
\begin{equation}
    f(x)\sin(\omega x) + g(x)\cos(\omega x), \label{oscfun}
\end{equation}
where $f(x)$ and $g(x)$ are assumed to be non-oscillatory functions. 

If one wishes to represent \eqref{oscfun} with a standard basis such as the Chebychev or Legendre polynomials, then the number of polynomials used must scale with $\omega$. Many papers have been written on computing the integrals of \eqref{oscfun}\cite{torii}\cite{methodsreview1}, and solving differential equations with oscillations\cite{petzjayyen1997}.

In the following paper, we will present an extension of orthogonal functions and their three term recurrence relations to problems involving oscillations. We will also provide a numerically stable method for computing the coefficients of the recursion for large $\omega$.


\section{Creating the basis}

For simplicity of formulas we focus on the case of \eqref{oscfun} where $\omega=2\pi k.$ We can justify this by noticing if $f(x)$ and $g(x)$ in \eqref{oscfun} are non-oscillatory, and $\omega=2\pi k+\epsilon$ where $|\epsilon|<\pi$, then there exist non-oscillatory $\hat{f}(x)$ and $\hat{g}(x)$ such that
\begin{equation}
    f(x)\sin(\omega x) + g(x)\cos(\omega x)=\hat{f}(x)\sin(2\pi k x) + \hat{g}(x)\cos(2\pi k x)
\end{equation}
via straightforward applications of addition formula for trigonometric functions. We define the following inner product

\begin{equation}
    <f(x),g(x)>:= \int_{-1}^1 f(x) g(x)  \;dx
\end{equation}

\begin{theorem}
The following functions form a basis for $\{x^k\sin(\omega x), x^k\cos(\omega x)\}_{k=0}^N$
\begin{eqnarray}
    p_0(x)&=&\cos(\omega x) \\
    q_0(x)&=&\sin(\omega x) \\
    p_1(x)&=& xp_0(x) +\frac{1}{2\omega}q_0(x)\\
    q_1(x)&=& xq_0(x)+\frac{1}{2\omega}p_0(x) \\
    p_{k+1}(x)&=& xp_k(x)-\frac{<xp_k,q_k>}{<q_k,q_k>}q_k(x)-\frac{<xp_k,p_{k-1}>}{<p_{k-1},p_{k-1}>}p_{k-1}(x) \\
    q_{k+1}(x)&=& xq_k(x)-\frac{<xq_k,p_k>}{<p_k,p_k>}p_k(x)-\frac{<xq_k,q_{k-1}>}{<q_{k-1},q_{k-1}>}q_{k-1}(x)
\end{eqnarray}

\end{theorem}

\begin{proof}
 We follow the standard proof for three term recurrence relations. First note that $p_k(x)$ is even when $k$ is even, and odd when $k$ is odd, while $q_k(x)$ is even when $k$ is odd, and odd when $k$ is even. It is straight forward to verify $p_0(x), q_0(x), p_1(x)$, and $q_1(x)$ are all orthogonal to each other. All that remains is to prove the state via induction. 
 
 Examine $<p_{k+1},p_j(x)>$ where $j<k-1$. Then $<xp_k,p_j>=<p_k,xp_j>=0$ due to the fact that $xp_j(x)$ is a polynomial of degree less than $k$, and by assumption, $p_k$ is orthogonal to all such polynomials.

\begin{align*}
    <p_{k+1},q_k>&=&<xp_k,q_k>-\frac{<xp_k,q_k>}{<q_k,q_k>}<xp_k,q_k>=0 \\
    <p_{k+1},p_k>&=&0 \\
    <p_{k+1},q_{k-1}>&=&0 \\
    <p_{k+1},p{k-1}>&=&<xp_k,p_{k-q}>-\frac{<xp_k,p_{k-1}>}{<p_{k-1},p_{k-1}>}<p_{k-1},p_{k-1}>=0
\end{align*}

Where the two middle lines are due to the even and odd properties. Similarly, $q_{k+1}(x)$ is orthogonal to $\{p_j(x),q_j(x)\}_{j=0}^k$. All that remains to be checked is the orthogonality of $p_{k+1}(x)$ and $q_{k+1}(x)$. However, one is even and the other is odd, so they must also be orthogonal.
\end{proof}

\section{Computing the Coefficients}

We now lay out a procedure for computing the coefficients of the recursion for $\omega$ large relative to $N$. First we must choose a basis to represent the orthogonal functions. Naively, one may want to use the basis $\{x^k\cos(\omega x), x^k\sin(\omega x)\}$ to represent the orthogonal basis. However, as $k$ increases, this basis becomes and more linearly dependent. This may be seen quickly by noting that the matrix whose coefficients are given by $H^\omega_{i,j}=<x^i\cos(\omega x),x^j\cos(\omega x)>$ converges to the coefficients of the infamously ill-conditioned Hilbert matrix divided by two as $\omega$ goes to positive infinity. Hence representing this space in the "monomial" basis leads to poor numerical accuracy. 

We will represent the $\{p_j(x),q_j(x)\}_{j=0}^k$ as products of the Legendre polynomials with sines and cosines. Indeed, one can see that in the limit as $\omega$ goes to infinity, $\{P_k(x)\cos(\omega x), P_k(x)\sin(\omega x)\}$ become orthogonal to each other. Hence for large $\omega$, one may expect the Legendre polynomial basis multiplied by sines and cosines is a good choice. 

\begin{eqnarray}
    <P_k(x)\cos(\omega x), P_j(x)\cos(\omega x)>=\frac{<P_k(x),P_j(x)>+<P_k(x),P_j(x)\cos(2\omega x)>}{2}
\end{eqnarray}

As $\omega$ goes to infinity, this converges to either zero when $j\neq k$ or $||P_k||^2/2$ when $j=k$. Using a known orthogonal basis to represent another has been used in \cite{Pric1979}.

It will be necessary to compute inner products of the form \\$<P_k(x)\cos(\omega x),  P_j(x)\cos(\omega x)>$, $<P_k(x)\cos(\omega x), P_j(x)\sin(\omega x)>,$ and\\ $<P_k(x)\sin(\omega x), P_j(x)\sin(\omega x)>$. We will develop a recursive algorithm for computing these coefficients. We examine the following.

\begin{eqnarray}
    <P_k(x)\cos(\omega x), P_j(x)\cos(\omega x)> \\
    =\frac{<P_k(x),P_j(x)>}{2}+\frac{<P_k(x),P_j(x)\cos(2\omega x)>}{2} \\
    =\frac{\delta_{kj}||P_k||^2}{2} +\frac{1}{2}\int_{-1}^1 P_k(x)P_j(x)\cos(2\omega x) \;dx \\
    =\frac{\delta_{kj}||P_k||^2}{2}+\frac{P_k(x)P_j(x)\sin(2 \omega x)}{2\omega}\Bigg|_{-1}^1\\
    -\frac{1}{4\omega}\int_{-1}^1 (P_k'(x)P_j(x)+P_k(x)P_j'(x))
    \sin(2\omega x) \; dx \\
    =\frac{\delta_{kj}||P_k||^2}{2}-\frac{1}{4\omega}\int_{-1}^1 \sum_{l=0}^{k-1-2l\geq0} (2(k-1-2l)+1)P_{k-1-2l}(x)P_j(x)\sin(2 \omega x) \; dx
    \\-\int_{-1}^1 \sum_{l=0}^{j-1-2l\geq0} \frac{(2(j-1-2l)+1)}{4\omega}P_{k}(x)P_{j-1-2l}(x)\sin(2 \omega x) \; dx \\
    =\frac{\delta_{kj}||P_k||^2}{2}-\frac{1}{4\omega} \sum_{l=0}^{k-1-2l\geq0} (2(k-1-2l)+1)<P_{k-1-2l}(x),P_j(x)\sin(2 \omega x)>
    \\-\frac{1}{4\omega}\sum_{l=0}^{j-1-2l\geq0} \frac{(2(j-1-2l)+1)}{4\omega}<P_{k}(x),P_{j-1-2l}(x)\sin(2 \omega x)>\qquad
\end{eqnarray}

Thus we have the inner product we would like to compute is the sum of previous inner products with a $\sin(2\omega x)$ instead of a $\cos(2\omega x)$. We define the following matrices

\begin{eqnarray}
    M1_{j,k}=<P_{j}(x),P_k(x)> \\
    M2_{j,k}=<P_{j}(x)\cos(\omega x),P_k(x)\sin(\omega x)> \\
    M3_{j,k}=<P_{j}(x)\cos(\omega x),P_k(x)\cos(\omega x)> \\
    M4_{j,k}=<P_{j}(x)\sin(\omega x),P_k(x)\sin(\omega x)> \\
    M5_{j,k}=<P_{j}(x),P_k(x)\cos(2\omega x)> \\
    M6_{j,k}=<P_{j}(x),P_k(x)\sin(2\omega x)> 
\end{eqnarray}

These matrices have the following relations

\begin{eqnarray*}
    M1_{j,k}=\delta_{ij}||P_k||^2 \\
    M2_{j,k}=\frac{M6_{j,k}}{2} \\
    M3_{j,k}=\frac{M1_{j,k}}{2}+\frac{M5_{j,k}}{2} \\
    M4_{j,k}=\frac{M1_{j,k}}{2}-\frac{M5_{j,k}}{2} \\
    M5_{j,k}=\frac{(1+(-1)^{j+k})\sin(2\omega)}{2\omega}-\frac{1}{2\omega}\sum_{l=0}^{j-1-2l\geq0}(2(j-1-2l)+1)M6_{j-1-2l,k} \\
    -\frac{1}{2\omega}\sum_{l=0}^{k-1-2l\geq0}(2(k-1-2l)+1)M6_{j,k-1-2l} \\
    M6_{j,k}=\frac{(-1+(-1)^{j+k})\cos(2\omega)}{2\omega}+\frac{1}{2\omega}\sum_{l=0}^{j-1-2l\geq0}(2(j-1-2l)+1)M5_{j-1-2l,k} \\
    +\frac{1}{2\omega}\sum_{l=0}^{k-1-2l\geq0}(2(k-1-2l)+1)M5_{j,k-1-2l}
\end{eqnarray*}

The matrices satisfy the following properties. All matrices are symmetric. M1 is diagonal and can be computed via the known norms of the Legendre polynomials. M2, M3, and M4 can all be computed once M5 and M6 are known. M5 and M6 can be populated by making entries on successive skew diagonals, i.e. first make $M5_{0,0}$ and $M6_{0,0}$. Then make $M5_{1,0}, M5_{0,1}, M6_{1,0},$ and $M6_{0,1}$ via the recursion. Continue by making the next skew diagonal formed of elements $M5_{j,k}$ and $M6_{j,k}$ such that $j+k=2$. Due to symmetry we only need to compute the upper halves of these matrices. The recursion relations are stable for $2\omega>j,k$. And finally, by assumption on $\omega$, $\sin(2\omega)=0$ and $\cos(2\omega)=1$. 

Let $f(x)$ and $g(x)$ have the forms

\begin{eqnarray}
    f(x)=\sum_{k=0}^N a_k P_k(x) \cos(\omega x) + b_k P_k(x) \sin(\omega x) \\
    g(x)=\sum_{k=0}^M c_k P_k(x) \cos(\omega x) + d_k P_k(x) \sin(\omega x)
\end{eqnarray},
then we have
\begin{equation}
    <f,g>=\vec{a}^T \cdot M2 \cdot \vec{d} + \vec{a}^T \cdot M3 \cdot \vec{c} + \vec{b}^T \cdot M2 \cdot \vec{c} + \vec{b}^T \cdot M4 \cdot \vec{d} 
\end{equation}
where every M matrix has been taken to have dimensions $N\times M$. In this framework we may compute all the coefficients of our recursion. However, it has been observed the norm of the "monic" orthogonal functions decays rapidly; it has been observed they decay roughly on an order of 2. Therefore it is recommended to compute the normalized orthogonal functions instead.

We only care about the case of large omega because when omega is small relative to N, you should just use regular quadrature. The determination for what large omega relative to N means will be from where the algorithm is stable. Particularly computing the matrices M1, M2, M3, M4, M5, M6, M7, and M8. The recurrence relation for computing these matrices is stable for $\omega>N$ where N is the size of the square matrix.

Now that we have a numerical method of computing inner products, we may represent the orthogonal functions in this basis and compute the coefficients of the recursion directly. The method will be stable as long as two conditions are met. The first is $\omega>N$ and the second is that the Legendre Polynomials multiplied by sines and cosines well approximates the space of our orthogonal functions.



\section{The Derivative Matrix}

We note that we may use the recurrence relation to compute integrals and derivatives of a given basis function. A given basis function may be represented as

\begin{equation}
    p_k(x)=\sum_{j=0}^k a_{kj}P_j(x)\cos(\omega x) + b_{kj}P_j(x)\sin(\omega x).
\end{equation}
By taking the derivative of both sides we arrive at
\begin{equation}
    p_k'(x)=\sum_{j=0}^k a_{kj}(P_j'(x)\cos(\omega x)-\omega P_j(x)\sin(\omega x)) + b_{kj}(P_j'(x)\sin(\omega x)+\omega P_j(x)\cos(\omega x))
\end{equation}.
Given the left hand side is in the form of polynomials multiplied by sines and cosines, we may represent it in our the basis of Legendre polynomials multiplied by sines and cosines. Hence we can form a derivative matrix. Let $\vec{b}=\{p_0(x),q_0(x),p_1(x),p_2(x),...$. Then we have
\begin{equation}
    \vec{b}'=\mathbf{D}\vec{b}
\end{equation}.
We note that $\mathbf{D}$ is triangular. 

Explicitly, the derivative matrix is

\begin{equation}
    \mathbf{I_1}=\begin{pmatrix}
1 & 0 \\
0 & 1 
\end{pmatrix}
\hspace{1cm}
    \mathbf{I_2}=\begin{pmatrix}
0 & 1 \\
-1 & 0 
\end{pmatrix}
\hspace{1cm}
\mathbf{0}=\begin{pmatrix}
0 & 0 \\
0 & 0 
\end{pmatrix}
\end{equation}

Then we have 

\begin{equation}
    \mathbf{D}=
    \begin{pmatrix}
        \omega \mathbf{I_2} & \mathbf{I_1}& \mathbf{0}& \mathbf{I_1}& \mathbf{0} & \mathbf{I_1} & \hdots \\
         & \omega \mathbf{I_2} & 3\mathbf{I_1} & \mathbf{0} & 3\mathbf{I_1} & \mathbf{0}& \hdots \\
         &  & \omega \mathbf{I_2} & 5\mathbf{I_1} & \mathbf{0} & 5\mathbf{I_1} & \hdots \\
         & & &  \omega \mathbf{I_2} & 7\mathbf{I_1} & \mathbf{0} & \hdots\\
         & & & & \hspace{-.2cm}\ddots & \hspace{-.2cm}\ddots & \hspace{-.2cm}\ddots 
    \end{pmatrix}
\end{equation}.
We note that this D Matrix is simple to understand, we have a block diagonal matrix from the derivative landing on sine and cosine, and then an upper triangular matrix that is directly similar to the derivative matrix for Legendre polynomials. This is the derivative matrix for taking the derivative of the $\{\cos(\omega x),\sin(\omega x), x\cos(\omega x), x\sin(\omega x), ...,x^N\cos(\omega x), x^N\sin(\omega x)\}$ basis. In order to get the derivative matrix for the orthogonal basis, let $\vec{f}, \mathbf{D}, and \vec{g}$ be in the first basis. Then we have

\begin{eqnarray*}
    \mathbf{D}\vec{f}=\vec{g} \\
    \mathbf{B}^{-1}\mathbf{DB}(\mathbf{B}^{-1}\vec{f})=(\mathbf{B}^{-1}\vec{g})
\end{eqnarray*}

Hence our derivative matrix in the orthogonal basis is $\mathbf{B}^{-1}\mathbf{DB}.$












\section{Future Directions}
To our knowledge, this is the first time orthogonal functions with a three term recurrence relation have been created in pairs. In order to create a quadrature method from these orthonormal functions, we must also generalize quadrature methods to handle the mixed three term recurrence relations. Computing the coefficients of these mixed three term recurrence relations was the necessary first step. 

We note this method can be generalized to include any number separate frequencies $\omega_1, \omega_2, ..., \omega_N$. The stability of the resulting numerical methods now depends on the smallest distance from one frequency to another relative to the number of orthogonal functions used. Of key interest will be the case where $\omega_1=0$ and $\omega_2$ is large. We believe this would lead to an Enriched Spectral Method. We also believe this method is also suitable and straight-forward to implement in higher dimensions on a square grid, as is standard for quadrature methods.

We are interested in applying this method to problems with singularities of differing orders such as $\{x^k,x^k\log(x),x^k\log(x)^2\}_{k=0}^N$. We believe we can develop methods suited to handling problems where these types of singularities (crack phenonmenons) arise. Certain papers have have orthogonal polynomials for for $\{x^k\log(x)^m\}_{k=0}^N$ for natural numbers $m$.



\bibliographystyle{plain}
\bibliography{refs}

\end{document}